\documentclass{article}

\usepackage{latexsym}
\usepackage{bbm}
\usepackage{amsmath,enumerate}
\usepackage[psamsfonts]{amssymb}
\usepackage{amsthm}
\usepackage[mathscr]{eucal}
\usepackage{hyperref}
\usepackage{graphicx}
\usepackage{color}		
\newtheorem{Theorem}{Theorem}[section]
\newtheorem{Proposition}{Proposition}[section]
\newtheorem{Lemma}{Lemma}[section]

\newtheorem{Remark}{Remark}[section]

\newtheorem{Definition}{Definition}[section]

\newcommand{\f}{\frac}
\newcommand{\1}{\mathbbm 1}
\newcommand{\p}{\partial}

\def\R{\mathbb{R}}
\def\NN{\mathbb{N}}
\def\N{\mathbb{N}}

\def\eps{\varepsilon}
\def\ep{\varepsilon}

\newcommand{\comJuan}[1]{\textcolor{black}{#1}}

\begin{document}

\title{Long-time asymptotics for polymerization models}

\author{Juan Calvo\footnote{{Departamento de Matem\'atica Aplicada and Excellence Research Unit ``Modeling Nature'' (MNat), Universidad de Granada}, 18071 Granada, Spain. Email adress: juancalvo@ugr.es}
\and Marie Doumic\footnote{Sorbonne Universit\'{e}, Inria, Universit\'{e} Paris-Diderot, CNRS,  Laboratoire Jacques-Louis Lions, F-75005 Paris, France. Email adress: marie.doumic@inria.fr}
\and Beno\^\i t Perthame\footnote{Sorbonne Universit\'{e}, Universit\'{e} Paris-Diderot, CNRS, Inria, Laboratoire Jacques-Louis Lions, F-75005 Paris, France. Email adress: benoit.perthame@upmc.fr}
}

 \maketitle
\begin{abstract} This study is devoted to the long-term behavior of nucleation, growth and fragmentation equations, modeling the spontaneous formation and kinetics of large polymers in a spatially homogeneous and closed environment. Such models are, for instance, commonly used in the biophysical community in order to model \emph{in vitro} experiments of fibrillation. We investigate  the interplay between four  processes: nucleation, polymerization, depolymerization and fragmentation.
We first revisit the well-known Lifshitz-Slyozov model, which takes into account only polymerization and depolymerization, and we show that, when nucleation is included, the system goes to a trivial equilibrium: all polymers fragmentize, going back to very small polymers. Taking into account only polymerization and fragmentation, modeled by the classical growth-fragmentation equation, also leads the system to the same trivial equilibrium, whether or not nucleation is considered.  Finally, when taking  into account a depolymerization reaction term, { we prove the existence of}  a steady size-distribution of polymers, as soon as polymerization dominates depolymerization for large sizes whereas depolymerization dominates polymerization for smaller ones - a case which
 fits the classical assumptions for the Lifshitz-Slyozov equations, but complemented with fragmentation so that  "Ostwald ripening" does not happen.
\end{abstract}

\noindent {\bf AMS Class. No.}  35B40; 35R09; 35Q92; 82D60

\section*{Introduction}

\subsection*{Framework and model}
The formation of large aggregates, polymers or fibrils, out of monomeric units, is a phenomenon of key importance in many application areas, from amyloid diseases to industrial processes. 
When the average size of polymers or aggregates is very large, a size-continuous framework is relevant and this is our framework: in the following, we denote $u(t,x)$ the concentration of polymers of size $x>0$ at time $t,$ and $V(t)$ the concentration of monomers at time $t$.

Assuming a closed and space-homogeneous environment - which is for instance the case for \emph{in vitro} experiments~\cite{Radford} -  the total mass needs to be conserved, \emph{i.e.}, we enforce the following equality:
\begin{equation}\label{eq:mass}
V(t) +\int_0^\infty x u(t,x) dx =V(0) +\int_0^\infty x u(0,x) dx  :=M>0, \qquad  \forall t\geq 0.
\end{equation}
Here, we have denoted by $M$ the total mass of monomers, present either under the monomeric form - the concentration $V(t)$ - or within polymers.
Note that $V$ is not directly homogeneous to $u,$ rather to $x u(t,x)dx;$ see e.g.~\cite{Collet,Doumic2009,Prigent} for further explanations  on the relation to a physical quantity appearing more explicitly in discrete models.

To describe the kinetics of polymers and monomers in an environment where polymers are too dilute to interact (no coagulation~\cite{EscoMischler2}), one classically considers four main reactions~\cite{Bishop}.
\begin{enumerate}
\item {Nucleation} is the formation of polymers out of momomers, by the spontaneous aggregation of monomers into a first stable - very small - polymer, called the \emph{nucleus}. Generally occurring with a very low rate, this reaction is of key importance in experiments where there are initially only monomers, but  becomes negligible as soon as enough polymers are formed, so that the other reactions dominate.
\item Polymerization is the growth in size of polymers by monomer addition. It is called a second-order reaction since the law of mass action assumes that it depends on the product of the concentration of monomers $V(t)$ by the concentration of polymers $u(t,x).$
\item Depolymerization is the decay in size of polymers by monomer loss (first-order reaction),
\item Fragmentation is the breakage of polymers into smaller polymers (first-order reaction).
\end{enumerate}
In the present study, we address  the question of the long-time behavior of models combining some or all of these mechanisms, with one main question underlying our study: is there a stable distribution of polymers, or do they 
dissociate into monomers? To give a comprehensive view on the interplay between these four reactions without burdening the text with technical details, we simplified the assumptions on the coefficients, which are designed to 
represent typical features rather than to be optimal.

To take into account polymerization, depolymerization and fragmentation,
a framework equation satisfied by the concentration of polymers $u(t,x)$ is
\begin{equation}
\label{eq:frame}
\begin{array}{l}
\displaystyle \frac{\partial u}{\partial t} + \frac{\partial}{\partial x} \big((V(t)-d(x)) u \big)+ B(x)u(t,x) =2  \int_x^\infty B(y) \kappa (y,x) u(t,y)\, dy,
\\  \\
u(0,x)=u_0(x) \geq 0.
\end{array}\end{equation}
Here, we  assume, for the sake of simplicity, a constant polymerization rate taken equal to $1$,  the corresponding term in the equation is $\f{\p}{\p x} (V(t) u(t,x))$. The depolymerization rate is denoted $d(x)\geq 0,$ whereas $B(x)\geq 0$ is the total fragmentation rate, and $\kappa(y,x)$ the fragmentation kernel, \emph{i.e.}, the probability measure on $[0,y]$ for polymers of size $y$ to give rise to polymers of size $x\leq y.$ 

If $V(t)-d(0)>0,$ a boundary condition at $x=0$ is needed. Therefore, we state
\begin{equation}
\label{bcB}
\big(V(t) -d(0) \big)u(t,0)\1_{V(t)-d(0)>0}=\varepsilon V(t)^{i_0} \1_{V(t)-d(0)>0},\quad i_0\ge 1,
\end{equation}
where $\1$ denotes the Heaviside function, $\varepsilon=0$ in the absence of nucleation and $\varepsilon=1$ to model the nucleation reaction. In this last case,  the nucleation reaction rate is taken to be $1,$ and $i_0\in \N^*$ represents the size of the nucleus, see~\cite{Prigent}.

To complement the model~\eqref{eq:frame}-\eqref{bcB}, we can either use the mass conservation~\eqref{eq:mass} or equivalently - as soon as all terms may be defined in 
appropriate spaces - by the following equation for the concentration of monomers $V(t)$:
\begin{equation}
\label{VeqA}
\frac{dV}{dt}= - V(t) \int_{0}^\infty  u(t,x) \, dx +\int_{0}^\infty d(x)u(t,x)dx, \qquad V(0)= V_0 \geq 0.
\end{equation}

\subsection*{Link with other equations}

The framework~\eqref{eq:mass}--\eqref{VeqA} embeds { as particular cases several instances of} two well-known models, the Lifshitz-Slyozov system and the non-linear growth-fragmentation equation. We review them briefly { and explain why we need to combine both in a suitable way to get reasonable models for large aggregates formation.}

\paragraph{Link with the Lifshitz-Slyozov system.}
First, when $B\equiv 0$ and $\ep=0$, \emph{i.e.} in the absence of fragmentation and nucleation, the system \comJuan{resembles} the well-known Lifshitz-Slyozov system~\cite{LifshitzSlyozov_1961}. Traditionally designed to model phase transition, the assumptions on the polymerization rate (here taken equal to $1$ 
) 
and the depolymerization rate $d(x)$ are such that no boundary condition at $x=0$ { is} required, the flux at zero being always going outward. Moreover, one of the key assumptions for phase transition models is that for large sizes, { polymerization} dominates { depolymerization}, { whereas for small sizes it is the reverse,} leading to larger and larger particles in smaller and smaller number, a phenomenon called "Ostwald ripening", see e.g.~\cite{FL3,LifshitzSlyozov_1961,Niethammer}. 

{ The original model for phase transitions by Lifshitz and Slyozov considers a polymerization rate  proportional to $x^{\f{1}{3}}$ together with a constant depolymerization rate~\cite{LifshitzSlyozov_1961}. It has been noted that for more general rate functions the dynamics of the size distribution of the clusters/polymers is driven by the size-dependency of the ratio}
$$
\f{\text{polymerization rate 
}
}{{\text{depolymerization rate}}}(x).
$$
{Therefore, the assumption "polymerization dominates depolymerization for large sizes, depolymerization dominates polymerization for the small sizes" can be conveniently phrased in terms  of the above ratio being (strictly) decreasing as a function of the size $x$. Note that in our case of constant polymerization rate this reduces to}
  \begin{equation}\label{as:d:dec}
d\in {\cal C}^1 (\R_+) \text{ is strictly decreasing}.
\end{equation}

Under this { monotonicity} assumption, it may be proved that either all the polymers depolymerize and $V(t)$ tends to $M$ (this is the case if there are not enough polymers initially, \emph{i.e.} $M<d(0)$), or $V(t)$ tends to $0$, the quantity of polymers $\int_0^\infty u(t,x)dx$ also tends to $0$, and the average size tends to infinity: for large times, all particles tend to be aggregated into only one cluster of infinite size. For finer results, we refer for instance to~\cite{NP4,NP2,NV1}.

\paragraph{Link with the growth-fragmentation equation.} When $d(x)\equiv 0$, \emph{i.e.}, if we neglect the depolymerization reaction, Equation~\eqref{eq:frame} turns out to be a non-linear case of the well-known growth-fragmentation equation, { which has been notoriously used, among other applications: (i) to describe cell division, (ii) as the underlying structure for the  
 so-called "prion model".}
 
{ It is well known that for the linear growth-fragmentation equation (\emph{i.e.} when $V$ is taken constant and $d=0$ in~\eqref{eq:frame}), the size distribution converges to the dominant eigenvector  and the population grows exponentially, with a rate of growth given by the dominant eigenvalue. This can be shown \emph{e.g.}  by means of generalized relative entropy techniques; we refer the interested reader to \cite{BP}. When adding the coupling with the monomers concentration $V$, our model may be seen as a variant of the so-called  prion model, introduced in~\cite{Greer} and  studied in~\cite{CL2,Pruss}. However,} this prion model describes \emph{in vivo} systems: nucleation is not considered, and monomers are permanently produced and degraded. In this way  the system is not closed: the conservation of mass~\eqref{eq:mass} is replaced by the following mass balance law
\begin{equation}\label{eq:mass:prion}
\f{d}{dt}\left( V(t) +\int_0^\infty x u(t,x) dx \right)=-\gamma V +\lambda, \qquad t\geq 0,
\end{equation}
for two constants $\lambda>0$ and $\gamma>0$ which represent respectively the production and degradation rates of monomers.
In that case, it has been proved that a non-trivial equilibrium solution $(\bar V, U(x))$  to~\eqref{eq:frame}, \eqref{bcB}, \eqref{eq:mass:prion} may exist,  with $d\equiv 0$ and $\ep=0,$ and this state is attractive under some assumptions on the coefficients~\cite{Gabriel_2015,Pruss}. Note however that the problem of convergence towards the equilibrium for general coefficients remains open.  

{ \paragraph{Combined framework: new possibilities.}

The Lifshitz-Slyozov framework as described above does not fit well to the application field we have in mind: for instance, amyloid fibrils  remain numerous instead of clustering all together, and their size, though very large, remains finite. One possibility in order to remedy this situation is 
to change the assumption~\eqref{as:d:dec}. We shall analyze this below; more precisely, we will be concerned with 
%
 %
%
the following specific form of growth:}
\begin{equation}
\label{as:d:inc}
d\in{\cal C}^1(\R_+,\R_+), \quad {\exists}\;\alpha, \beta >0,\quad\quad 0<\alpha \le d'(x)\le \beta.
\end{equation}

{ Our analysis in Section \ref{sec:LS} will show that we do not get a satisfactory picture in this way. Another possibility to fix this situation is to add an extra reaction - nucleation or fragmentation - in the system, be it in combination with \eqref{as:d:dec} or \eqref{as:d:inc}. This leads eventually to the analysis of augmented versions of the growth-fragmentation framework (and more specifically that of the prion model) discussed right above. However,  the context of this paper is different than that of the standard prion model:  we assume mass conservation, that means $\lambda=\gamma=0$ in Equation~\eqref{eq:mass:prion}, which corresponds for instance to \emph{in vitro} experiments}. One can see in Equation~\eqref{VeqA} that in the absence of depolymerization the quantity of monomers can only decrease, being consumed by polymerization. But if $V(t)$ goes to zero, then the polymerization rate vanishes in Equation~\eqref{eq:frame}, so that for large times it is close to the pure fragmentation equation, for which it is well-known that $u(t,x)$ tends to a Dirac mass at $0$, all polymers  become infinitely small. { These qualitative considerations are made rigorous in Theorem~\ref{th:probB:asymp2:2}.}

{ Such "instability" of polymers, \emph{i.e.} the fact that their sizes, after having increased by nucleation and growth-fragmentation, would go back to dusts, }does not correspond to experimental observations \comJuan{\cite{Radford}}. { T}he prion model thus needs to be enriched by another reaction in order to obtain a steady distribution of polymers. { Our contribution in this paper is to augment the prion model with a depolymerization reaction, which can be thought of as a combination of both Lifshitz-Slyozov and growth-fragmentation frameworks.}

\

The paper is organized as follows. In the next section, we state our main results and give a rigorous meaning to the previous  qualitative considerations. The subsequent sections are used to prove these results.

\section{Main results}

\subsection{Notations and framework assumptions}
\label{sec:amr}
Throughout the paper, the domain for both the space and time variables is $\R_+:=[0,\infty)$. In case of possible ambiguity, subscripts are used  to denote the functional spaces and indicate to which variable $x$ or $t$ they refer to. For instance, $L_t^1$ is the space of Lebesgue integrable functions of time defined over $\R_+$ and similarly for $L_x^1$; note that we omit the base set, which by default is $\R_+$ unless otherwise stated. Similarly, standard Sobolev spaces are denoted as $W_t^{1,1},\, W_x^{1,2}$ and so forth. Sometimes we  use weighted spaces like $L^1 \big(\R_+, (1+x^2) dx \big)$ with obvious notations, or for instance $L^1 \big(\R_+, (1+x^2) dx \big)^+$ for the cone of nonnegative functions of $L^1 \big(\R_+, (1+x^2) dx \big)$. We also make use of spaces of point-wise defined functions, ${\cal C}(\R_+)$ for continuous functions, ${\cal C}^1(\R_+)$ for differentiable functions with continuous derivatives and ${\cal C}^1_b(\R_+)$ for differentiable functions with continuous and bounded derivative.

We recall that the p-Wasserstein distance between two probability measures $\mu, \nu$  (or two nonnegative bounded measures with the same total mass)  is defined as
$$
W_p(\mu,\nu):= \left(\inf_{\eta \in \Gamma (\mu,\nu)} \int_{(0,\infty)^2} |x-y|^pd\eta (x,y)\right)^{1/p}\,, \quad 1\le p <\infty,
$$
with $\Gamma (\mu,\nu)$ the set of measures on $(0,\infty)^2$ with marginals $\mu$ and $\nu$. Note that
$$
W_p(\mu,\nu) \le \left(\int_{(0,\infty)^2} |x-y|^p d\mu(x)d\nu(y)\right)^{1/p}.
$$
\noindent
We define the moments of $u(t,\cdot)\in L_x^1(0,\infty)$ as
\begin{equation}\label{def:rho}
M_{n}(t):= \frac{1}{n}\int_0^\infty |x|^n u(t,x)\, dx,\quad n>0, \qquad \rho(t):=\int_0^\infty u(t,x)dx.
\end{equation}
Notations $M$ for the total mass defined in~\eqref{eq:mass} and $\rho$ for the total number of polymers are used throughout the document (not to be confused with $M_n$).

\medskip

We notice that a boundary condition is needed for \eqref{eq:mass}--\eqref{eq:frame} to be well posed only when $V(t) >d(0)$; we have stated~\eqref{bcB} in a way that makes sense even when no boundary condition is needed. This also means that a minimal number of monomers is needed for nucleation reactions to take place under our current formulation. We thus require the initial number of monomers to be large enough so that we avoid trivial dynamics,
\begin{equation}
\label{V_init}
V_0>d(0) \geq 0.
\end{equation}
Under Assumption~\eqref{as:d:inc}, there is no loss of generality in assuming that \eqref{V_init} holds, as the following proposition shows.
\begin{Proposition}
\label{pr:marginal_cases}
Let $(V,u)\in {\cal C}_b^1(\R_+)\times {\cal C}\left(\R_+,L^1\left((1+x^2)dx\right)\right)$ be any nonnegative solution of ~\eqref{eq:mass} and \eqref{VeqA} such that the initial datum verifies $V_0 <d(0)$. Assume that $d\in {\cal C}^1(\R_+)^+$ satisfies \eqref{as:d:inc}. Then the following statements hold true:
\begin{enumerate}
\item If $M>d(0)$  there is a $t^*\in (0,\infty)$ such that $V(t^*)=d(0)$ and $\frac{d V}{dt}(t^*)>0$.
\item If $M\le d(0)$ then $\vert M-V (t) \vert \leq(M-V_0)e^{-\alpha t}$.
\end{enumerate}
\end{Proposition}
\begin{proof} Using \eqref{VeqA},  \eqref{as:d:inc} and the mass conservation~\eqref{eq:mass}, we have 
$$
\frac{dV}{dt} \ge \int_0^\infty (d(0)+\alpha x -V(t))u(t,x)\, dx=\alpha (M-V(t)) + \left(d(0)-V(t)\right) \rho(t).
$$
Hence, we may write 
$$
\frac{dV}{dt} \ge \alpha (M-V(t)) \quad \mbox{as long as}\quad V(t) \leq d(0)
$$
which implies
$$
V(t) \ge M + (V_0-M) e^{-\alpha t} \quad \mbox{as long as}\quad V(t) \leq d(0).
$$
Therefore, if $M>d(0)$ we are able to find some $0<t^*<\infty$ such that $V(t^*)= d(0)$ and $\frac{dV(t^*)}{dt} >0$. If $M\le d(0)$ then by mass conservation again we deduce that $\lim_{t\to \infty} V(t) = M$.
\end{proof}
The former result means that when depolymerization rates increase with size, if we start with a low monomer number, either we fall into the regime given by \eqref{V_init} in finite time or the dynamics is somewhat trivial: the polymerized mass vanishes completely on the long time run. These considerations play no role when depolymerization rates decrease with size, as we  see it in Section~\ref{sec:steady}.

Under the assumptions of our study, we take for granted that, for $u_0\in L^1 \big(\R_+, (1+x^2) dx \big)^+$ and  $V_0\ge 0$, there exists a  nonnegative weak solution $(V,u)\in  {\cal C}^1 (\R_+) \times {\cal C}\big(\R_+,L^1((1+x^2)dx)\big)$ to the system~\eqref{eq:mass}, \eqref{eq:frame} and \eqref{bcB}. Such a result is a work in preparation by J. Calvo - \comJuan{see also \cite{Deschamps2017}}; it may be obtained along the lines of \cite{Collet2000, Laurencot2001, Laurencot2007, Simonett}.
%
This framework allows us to state the main results of the document. Their guideline is the following: considering some or all of the four main reactions described in the introduction - nucleation, polymerization, depolymerization and fragmentation - and modeled by the framework system~\eqref{eq:mass}, \eqref{eq:frame}, \eqref{bcB}, what is the asymptotic behavior of the polymers and monomers $\left(u(t,x),V(t)\right)$? Which reaction rates may lead toward a steady size distribution of monomers?

\subsection{Lifshitz-Slyozov revisited}
\label{sec:main_results}

First, we concentrate on polymerization and depolymerization reactions, \emph{i.e.} the Lifshitz-Slyozov system, when $B\equiv 0$ in~\eqref{eq:frame} and $\ep=0$ in~\eqref{bcB}. This gives the system

\begin{equation}
\label{LS} 
\left\{
\begin{array}{ll}
&\frac{dV}{dt}= - V(t) \int_{0}^\infty  u(t,x) \, dx +\int_{0}^\infty d(x)u(t,x)dx, \qquad V(0)= V_0 ,
\\ \\ 
&\frac{\partial u}{\partial t} + \frac{\partial}{\partial x} ((V(t)-d(x)) u) =0,\qquad u(0,x)=u_0(x), \\ \\
&\big(V(t) -d(0) \big)u(t,0)\1_{V(t)-d(0)>0}=0.
\end{array}\right.
\end{equation}
Under the assumption~\eqref{as:d:dec} of a decreasing depolymerization rate, we refer to~\cite{Niethammer,NP4,NP3} for results showing the "Ostwald ripening" if $M>d(0)$. In this section, we explore the reverse case to investigate the possibility of  a {steady} distribution, hence we work under  the reverse assumption~\eqref{as:d:inc} of an increasing depolymerization reaction.

The equation for $u$ in~\eqref{LS} is a non-linear transport equation, which asymptotic is closely related to the characteristic curves as classically defined below.
\begin{Definition} [Characteristics]
Given $z\in [0,\infty)$, let $X :[0,\infty)^2 \rightarrow [0,\infty)$ be the ${\cal C}^1$ solution  of 
$$
\frac{d}{dt}X(t,z) = V(t)-d(X(t,z)),\quad X(0,z)=z.
$$
\label{def:Characteristics}
\end{Definition}
\begin{Remark}
Using \eqref{as:d:inc} we get
$$
\frac{d X(t,z)}{dt} \le V(t) -\alpha X(t,z),
$$
which implies that the characteristics always remain bounded:
$$
X(t,z)\le e^{-\alpha t} \left(z+\int_0^t V(\tau)e^{\alpha \tau}\, d\tau \right) \le  z + \f{M}{\alpha}.
$$
\label{rm:char_bound}
\end{Remark}

Using the characteristic curves, we obtain the following asymptotic result.
\begin{Theorem} [Lifshitz-Slyozov system: concentration at critical size]

Let $d$ satisfy~\eqref{as:d:inc}, $V_0$ satisfy \eqref{V_init} and $u_0\in L^1(\R_+,(1+x^2)dx)^+$ with $\rho_0=\int_0^\infty u_0(x)dx>0$.  Let  $M$ be defined by~\eqref{eq:mass}.

\
There exists a unique solution $\bar x>0$ to the equation
\begin{equation}\label{eq:asymp}
M=\rho_0 \bar x + d(\bar x),
\end{equation}
and the  solution $(V,u)\in {\cal C}^1(\R_+)\times {\cal C} \big(\R_+,L^1((1+x^2)dx) \big)$ to the Lifshitz-Slyozov system~\eqref{LS}
satisfies 
\begin{enumerate}
\item For all $z\geq 0,$ 
$$\int_0^\infty \vert X(t,z)-x\vert^2 u(t,x) dx \leq e^{-2\alpha t} \int_0^\infty \vert z-x\vert^2 u_0(x) dx, $$
\label{th:probA:asymp1:1}
\item $\lim_{t\to \infty} V(t)=\bar V:=d(\bar x)$, 
\label{xbar:1} $\qquad \forall\; z\geq 0,\quad \lim_{t\to \infty} X(t,z) = \bar x$, 
\item $u(t,x)$ converges to $\rho_0 \delta_{\bar x}$  exponentially fast in the sense of the Wasserstein distance: for some constant $C>0$ we have
$$
W_2(u(t,\cdot),\rho_0 \delta_{\bar x}) \leq C e^{-\alpha t},\qquad \vert V(t)-d(\bar x)\vert \leq C e^{-\alpha t}.
$$
\label{xbar:3}
\end{enumerate}
\label{th:probA:asymp1}
\end{Theorem}
The proof of Theorem~\ref{th:probA:asymp1} is detailed in Section~\ref{sec:LS}. It relies on two entropy inequalities, inspired by~\cite{Michel2008} and~\cite{Collet2002} respectively { (we call  \comJuan{ in the sequel} an "entropy" a functional which decreases along solutions {- i.e. a Lyapunov functional, } and not a physical entropy)}. This result shows that the solution of the Lifshitz-Slyozov system \comJuan{\eqref{LS}}, taken with the opposite assumption on the ratio polymerization/depolymerization as the standard case, converges to a singular steady state $(\bar V, \rho_0 \delta_{\bar x})$, with $\bar x$ such that the decreasing total growth/transport rate $\bar V-d(x)$ vanishes at $\bar x$, this rate being positive  for $x<\bar x$ and negative for $x>\bar x.$ 

Such a result is not
observed in experiments, because the Lifshitz-Slyozov equation is a first-order approximation of the "true" discrete system, the so-called Bekker-D\"oring system, see~\cite{Ball-Becker-Doring,BeckerDoring_1935}. At points like $\bar x,$ where the total transport rate vanishes, the second order correction, a diffusion term~\cite{ColletHariz}, would dominate, 
changing the steady state from a Dirac to a size-continuous distribution.

Let us now add a nucleation reaction in the system. Specifying $\ep=1$ in the boundary condition~\eqref{bcB}, we obtain the  system
\begin{equation}
\label{LSN} 
\left\{
\begin{array}{ll}
&\frac{dV}{dt}= - V(t) \displaystyle \int_{0}^\infty  u(t,x) \, dx +\int_{0}^\infty d(x)u(t,x)dx, \qquad V(0)= V_0 ,
\\ \\ 
&\frac{\partial u}{\partial t} + \frac{\partial}{\partial x} ((V(t)-d(x)) u) =0,\qquad u(0,x)=u_0(x), \\ \\
&\big(V(t) -d(0) \big)u(t,0)\1_{V(t)-d(0)>0}=V(t)^{i_0} \1_{V(t)-d(0)>0}.
\end{array}\right.
\end{equation}
Under the usual assumption~\eqref{as:d:dec} of a decreasing $d$, there would be no difference with the usual result of Ostwald ripening, because after a while $V(t)<d(0)$ and nucleation does not act  any longer. This is stated in the following theorem, still for increasing $d$ given by the assumption~\eqref{as:d:inc}.

\begin{Theorem} [Lifshitz-Slyozov system with nucleation]
\label{th:probB:asymp2}

With the assumptions and notations of Theorem~\ref{th:probA:asymp1}, let $i_0\in \N^*.$
The  solution $(V,u)\in {\cal C}^1(\R_+)\times {\cal C} \big(\R_+,L^1((1+x^2)dx)\big)$ of the Lifshitz-Slyozov system with nucleation~\eqref{LSN} satisfies
\begin{equation}\label{asymp:LSN}
\left\{ 
\begin{array}{l}
 \displaystyle \lim_{t\to\infty} \int_0^\infty u(t,x) dx=+\infty, \quad \lim_{t\to\infty} V(t)=d(0), 
\\ \\
\lim_{t\to \infty} x u(t,x) = \big(M-d(0)\big) \delta_0(x),  \qquad \text{weakly in measures}.
\end{array} 
\right.
\end{equation}
More precisely, we have the following convergence rates.
\begin{enumerate}

\item If $d(0)>0$, then 
$$\lim _{t\to\infty} \frac{\rho(t)}{t} = d(0)^{i_0}, \qquad \lim_{t\to \infty} t (V(t)-d(0)) = \f{d'(0)}{d(0)^{i_0}}(M-d(0)).$$
\label{th:pbB:extra_info:d(0)>0}
\item If $d(0)=0$, then 
$$
\lim_{t\to \infty} \frac{\rho(t)}{t^{\f{1}{i_0+1}}} =(1+i_0)^{\f{1}{i_0+1}}\left(d'(0)M\right)^{\f{i_0}{i_0+1}},
\quad 
 \lim_{t\to \infty} t^{\f{1}{i_0+1}} V(t) = \left(\frac{d'(0)M}{1+i_0}\right)^{\f{1}{i_0+1}}.
$$
\label{th:pbB:extra_info:d(0)=0}
%
%
\end{enumerate}
\end{Theorem}
This result is proved in Section~\ref{sec:LSN}, using strongly the assumption~\eqref{as:d:inc} in successive estimates. This shows a  destabilization effect of nucleation: rather than increasing the whole solution, as one could first guess, it  leads to a complete depolymerization and to a mass concentration around zero. The  explanation is that under these assumptions, since we always have $V(t)>d(0),$ the nucleation reaction permanently fuels the total number of polymers, which increases to infinity, and since the total mass remains finite, the only possibility is that the average size of the polymers vanishes.

\subsection{Complete model: a possible {steady state}}
\label{sec:frag}
The previous results take into account the three reactions of polymerization, depolymerization and nucleation, and show that these reactions alone cannot render out the spontaneous creation - by nucleation - of { a steady distribution of} fibrils: either they all collide in a unique infinitely large aggregate (Ostwald ripening) or they disintegrate into dusts (shattering). We thus consider now also fragmentation, sometimes denominated  in the biophysics literature a \emph{secondary pathway}~\cite{Bishop}, \emph{i.e.} a reaction { with a much} smaller reaction rate than the primary pathway (here, the polymerization). Nevertheless, it is able to play a prominent role to modify the primary reactions. 
To simplify our analysis, we  assume that the fragmentation rate $B$ and binary fragmentation kernel $\kappa$ satisfy 
\begin{equation}
\label{customary}
\int_0^y \kappa(y,dx)=1, \quad \int_0^y x\, \kappa(y,dx)=y/2,
\end{equation}
\begin{equation}
\label{Brates:1}
\exists\;B_m>0,\qquad B(x)\geq B_m>0\quad \forall x\ge 0,
\end{equation}
which provides us with a convenient control on the relative size of fragmentation rates. This  bound from below is somehow a strong assumption that we can relax.

Two very different behaviors occur. First, under the assumption of an increasing - or equal to $0$ - depolymerization rate, with or without nucleation, the fragmentation strongly amplifies the previous result of dust formation stated in Theorem~\ref{th:probB:asymp2}. This is stated in the following theorem.
\begin{Theorem}
With  the assumptions and notations of Theorem~\ref{th:probA:asymp1}, let $\ep\in\{0,1\}$, $i_0\in \N^*,$ let  $B(x)$ and $\kappa(y,x)$ satisfy~\eqref{customary}--\eqref{Brates:1}. Then the  solution $(V,u)\in  {\cal C}^1(\R_+)\times {\cal C} \big(\R_+,L^1((1+x^2)dx)\big)$ to the system~\eqref{eq:mass}\eqref{eq:frame} \eqref{bcB} satisfies~\eqref{asymp:LSN}, and more precisely, for $t>0$ we have 
$$
\int_0^\infty u(t,x) dx \ge e^{B_mt}\int_0^\infty u_0(x) dx \quad \mbox{{\text{and}}}\quad V(t)-d(0)=O(t e^{-B_mt})>0.
$$
\label{th:probB:asymp2:2}
\end{Theorem}
We prove this theorem in Section~\ref{sec:frag:instab},  following the same lines as for Theorem~\ref{th:probB:asymp2}.

\

To avoid dust formation, we may use our assumption~\eqref{as:d:dec} of a decreasing depolymerization rate, which we precise as   follows:

\begin{equation}\label{as:d}
\begin{array}{c}
d\in {\cal C}^1 (\R_+,\R_+^*)^+ \text{ is strictly decreasing},
\\ \\
{\exists}\, C_d>0,\; n\in \NN \backslash \{0\}, \quad  \forall\; x\geq 1,\qquad d(x) -  d(\infty) \geq  C_d x^{-n} .
\end{array}
\end{equation}

As already  mentioned, with this assumption solutions of  the original Lifshitz-Slyozov system undergo Ostwald ripening. When adding fragmentation to the system, the effect is to 
generate a non-trivial steady state to emerge, where small polymers feed large ones, and large ones in turn break down to feed smaller ones.  The precise result is stated in  Theorem~\ref{th:steady} under  additional assumptions. 

First, the fragmentation kernel should not charge $0$ or $1$ exclusively, a property which  is expressed by the following standard assumption 
\begin{equation}
\forall x\geq 0, \quad  a_k(x):= 1- 2\int_0^x \kappa(x,y)\f{y^2}{x^2} dy \geq c>0.
\label{as:frag}
\end{equation}
Second, we replace the former assumption~\eqref{Brates:1} on the fragmentation rate by\begin{equation}
\label{Brates:2}
\exists {\cal A}>0, \, B_m>0\quad \mbox{such that}\quad B(x\geq {\cal A}) \geq B_m >0,
\end{equation}
\begin{equation}
\label{Brates:3}
\sup_{x \geq 0}  B(x):= B_M < \infty.
\end{equation}
We also need that the fragmentation rate $\kappa$ vanishes for small sizes depending on $B$, namely
\begin{equation}\left\{\begin{array}{l}
\label{Bandk}
\exists \; \gamma >0,\;  \qquad  B(x)x^{-\gamma} \in L^\infty((0,\infty)),
\\[10pt]
{ \forall \; x_{\min}>0,\; \exists\; C>0\text{ such that}}\\[10pt]
\forall \,x_0<x_{\min},\, \big \vert \displaystyle\int_x^{x_0} \kappa(y,z) dz \big \vert \leq \min \big (1, C \f{\vert x-x_0\vert^\gamma}{y^\gamma} \big)\quad \forall \, x,\,y > 0 .
\end{array} \right. \end{equation}
\begin{Theorem} [Stable fibrils]
Let $d \in {\cal C}^1(\R_+)$ be a nonnegative decreasing function satisfying \eqref{as:d}.
Assume that \eqref{customary} and \eqref{as:frag}--\eqref{Bandk} hold true. Then, there exists a positive steady state  $(\bar V,U(x))\in \R_+^* \times L^1((1+x^2)dx)$  solution to \eqref{eq:mass}--\eqref{bcB}, with $\inf_x d(x) <\bar V < d(0)$. 
\label{th:steady}
\end{Theorem}
This theorem is proved in Section~\ref{sec:steady}. { It states the existence of a positive steady state solution of the polymerization-depolymerization-fragmentation problem; however, the question of whether the solution converges towards this steady state remains open. The same question remains open for a large variety of nonlinear versions of the growth-fragmentation equation, such as the prion model, for which asymptotic behaviour has been solved only for some specific cases, see~\cite{Gabriel_2015}.}

\section{Lifshitz-Slyozov system revisited}
\label{sec:LS}

This section is devoted to  the proof of Theorems \ref{th:probA:asymp1} and~\ref{th:probB:asymp2}.
 We first  state some additional properties of the solutions.
\begin{Lemma}
Let $d$ satisfy~\eqref{as:d:inc}, $V_0$ satisfy~\eqref{V_init} and  $u_0\in L^1\left((1+x^2)dx\right)^+$. The solution $(V,u)\in {\cal C}^1(\R_+)\times {\cal C}\left(\R_+,L^1 \left((1+x^2)dx\right)^+\right)$  either of the system \eqref{LS} or of~\eqref{LSN} satisfies
\begin{enumerate}
\item For all times $t\geq 0$, we have $d(0)< V(t) \leq M$.\label{lem:probA:exist:1}
\item If $M_\theta(0)<\infty$ for some $\theta >0$, then $M_\theta(t)$ is bounded uniformly in time.
\item For system~\eqref{LS}, the total number of polymers $\rho(t)$ is constant.
\end{enumerate}
\label{lem:probA:exist}
\end{Lemma}
\begin{proof}
The bound $V(t)\leq M$ is immediate for any nonnegative solution and follows from the mass conservation~\eqref{eq:mass}.
For the lower bound on $V(t),$ we notice that, because $d$ is increasing, 
$$\begin{array}{ll}
\f{d \big(V(t)-d(0) \big) }{dt}  &=  -\rho(t) \left(V(t)-d(0)\right) + \displaystyle \int_0^\infty \big (d(x) -d(0)\big) u(t,x) dx 
\\ \\
&\geq - \rho(t) \left(V(t)-d(0)\right).
\end{array}
$$
Therefore, we conclude that 
$$
V(t)\ge d(0)+(V_0-d(0))e^{-\int_0^t \rho(s) ds } > d(0).
$$
For the second item, we may use the characteristics to represent solutions as
\begin{equation}
\label{repr_formula}
u\left(t,X(t,x)\right)= u_0(x) \exp \left\{ \int_0^t d'(X(\tau,x))\, d\tau\right\}.
\end{equation}
We also note that
\begin{equation}
\label{Jac}
\frac{\partial X(t,z)}{\partial z} = \exp \left\{-\int_0^t d'(X(\tau,z))\, d\tau�\right\}.
\end{equation}

We can extend $d(x)$ by the constant $d(0)$ for $x<0$, and thus define the characteristic in $\R_-$, then we define  $\bar z (t)$, the value such that $X(t,\bar z)=0$. Changing variables, using \eqref{Jac} and then in \eqref{repr_formula}, we find 
$$\begin{array}{rl}
\displaystyle \int_0^\infty x^\theta u(t,x)\, dx &= \displaystyle \int_{\bar z(t)}^\infty X(t,z)^\theta u(t,X(t,z)) \frac{\partial X(t,z)}{\partial z} \, dz
\\[10pt]
&= \displaystyle \int_{\bar z(t)}^\infty X(t,z)^\theta u(t,X(t,z)) \exp \left\{-\int_0^t d'(X(\tau,z))\, d\tau\right\} \, dz
\\[10pt]
&= \displaystyle  \int_{\bar z(t)}^\infty X(t,z)^\theta u_0(z) \, dz \le \int_0^\infty \left(z+ M/\alpha \right)^\theta u_0(z)\, dz,
\end{array}$$
where we used Remark \ref{rm:char_bound} for the last step. 

For the third item, the conservation of the total number of polymers $\rho(t)$ is obtained by integrating the equation for $u$ and using  the boundary condition at $x=0$, thanks to the fact that $V(t)>d(0).$
\end{proof}

\subsection{Lifshitz-Slyozov without nucleation}
\label{subsec:LS}

Theorem \ref{th:probA:asymp1} follows essentially from two different entropy inequalities. The first one, inspired from~\cite{Michel2008},  shows the exponential convergence of the mass along any characteristic curve, which in turns implies that all characteristic curves converge to each other. The second entropy inequality  is an adaptation of the entropy functional introduced by \cite{Collet2002} in the context of the Lifshitz--Slyozov model, and shows the convergence of $V(t)$.
\begin{Lemma}[Entropy inequality - convergence of the characteristic curves~\cite{Michel2008}] \label{lem:entropy1}
Let $d$ satisfy~\eqref{as:d:inc}, $u_0\in L^1\left((1+x^2)dx\right)^+$ and $V_0>d(0)\geq 0.$ Let $(V,u)\in {\cal C}^1(\R_+)\times {\cal C}\left(\R_+,L^1 \left((1+x^2)dx\right)^+\right)$ be a solution of \eqref{LS}. 
Let us define, for any  $z\geq 0,$
$$
g(t,z):=\int_0^\infty u(t,x)\vert X(t,z)-x\vert ^2\, dx.
$$
We have
$$
g(t,z)\leq g(0,z) e^{-2\alpha t}.
$$
\end{Lemma}

\begin{proof}
Because $u$ has a finite second moment, $g$ is well defined and we may also define
$$
\tilde g(t,z):=\int_0^\infty u(t,x)(X(t,z)-x) \big(d(X(t,z))-d(x)\big)\, dx \geq \alpha g(t,z).
$$
An immediate calculation gives 
%
$$\begin{array}{rl}   \displaystyle \frac{d g}{dt}  (t,z) \! \! \! \!
&= \displaystyle\int_0^\infty \biggl(\f{\p}{\p t} u(t,x) \vert X(t,z)-x\vert ^2\  + 2 u(t,x)  \f{dX}{dt} (t,z)\big(X(t,z) -x \big) \biggr) dx
\\[15pt]
&= \displaystyle\int_0^\infty  \biggl( \f{\p }{\p x} \big(u\big(d(x) - V(t) \big)\big) \vert X-x\vert ^2\   + 2 u  (V(t) - d(X)) \big(X -x \big) \biggr) dx
\\[15pt]
&=  \displaystyle\int_0^\infty \biggl(2u(d(x) - V(t) )  ( X-x)\  + 2 u  (V(t) - d(X)) \big(X -x \big) \biggr) dx
\\[15pt]
&= - 2 \tilde g(t,z) \leq -2 \alpha g (t,z),
\end{array}
$$
and hence $g (\cdot, z)  \in L^1_t (0,\infty)$ with the announced decay.
\end{proof}
This proves the point~\ref{th:probA:asymp1:1} in Theorem~\ref{th:probA:asymp1}, and shows that the mass concentrates along any characteristic curve. To obtain the convergence of the characteristics towards a fixed point, we use a second entropy inequality, directly adapted from Collet {\em et al}~\cite{Collet2002},  
\begin{Definition}[\cite{Collet2002}]
\label{def:entropy2}
Let $k:[0,\infty)\rightarrow [0,\infty)$ be a ${\cal C}^1$ function. 
We introduce
$$
H_k(t):= \int_0^\infty k(x) u(t,x)\, dx + K(V(t)),\quad K(v)=\int_{d(0)}^v k'(d^{-1}(s))\, ds.
$$
\end{Definition}
Note that this definition makes sense only if $V(t)$ lies in  the range of $d$, \emph{i.e.} under assumption~\eqref{as:d:inc}, the set $[d(0),\infty)$, which is the case thanks to Lemma~\ref{lem:probA:exist}.
\begin{Lemma} [Entropy inequality -  adapted from~\cite{Collet2002}]
\label{lem:entropy2}
Let $(V, u)$ be a solution for either System~\eqref{LS} or~\eqref{LSN}. For $k$ a ${\cal C}^1$ convex positive function such that $\int_0^\infty k(x) u_0(x) dx <+\infty$, with furthermore $k(0)=0$ when $(V,u)$ is solution to~\eqref{LSN}. Then  $H_k(t)$ is well-defined at any time  and we have
$$
\frac{d}{dt}H_k(t)=\int_0^\infty u(t,x) (V(t)-d(x)) (k'(x)-k'(d^{-1}(V(t)))\, dx \leq 0.
$$
\end{Lemma}
\begin{proof}
For the sake of completeness, we recall the proof from~\cite{Collet2002}. We write 
$$\begin{array}{ll}
\frac{d}{dt}H_k(t)&=  \displaystyle \int_0^\infty k(x) \f{\p}{\p t} u(t,x)\, dx + \f{dV}{dt} k'(d^{-1}(V(t)))
\\ \\ 
&=  \displaystyle \int_0^\infty k(x) \f{\p}{\p x}\big( (d(x)-V(t) )u \big)\, dx + k'(d^{-1}(V(t)))  \int_0^\infty (d(x) - V(t))u dx
\\ \\
&=  \displaystyle \int_0^\infty  \big(d(x)-V(t)\big) u(t,x)\big( k'(d^{-1}(V(t)))-k'(x)\big) dx. 
\end{array}
$$
The negativity follows because the mapping $x \mapsto k' \circ d^{-1}$ is increasing. 
\end{proof}
\begin{Remark}
Due to the boundary condition, the  entropy inequality of Lemma~\ref{lem:entropy1}, which shows the concentration of mass along any characteristic curve, fails for the nucleation boundary condition (\eqref{bcB} with $\ep=1$) of System~\eqref{LSN}, while the entropy  inequality of Lemma~\ref{lem:entropy2},  remains true for both models.
\end{Remark}

\begin{Remark}
Our situation  is in some sense the opposite of the one analyzed in \cite{Collet2002} for the classical setting of the Lifshitz--Slyozov model: large clusters grow larger as time advances whereas small clusters tend to become even smaller. This explains why we have an entropy inequality for convex functions $k,$ whereas it is obtained for concave functions $k$ in \cite{Collet2002}.
\end{Remark}

\begin{proof}[Proof of Theorem~\ref{th:probA:asymp1}]

The proof of convergence, both for $V$ and for the characteristic curves, combines both entropy inequalities. We use $k(x)=\int_0^x d(s)\, ds$ in  Lemma~\ref{lem:entropy2} and follow the following steps.
\\[5pt]
{\em Step 1.} For the entropy built on $K(\cdot)$ as mentioned above, using that $\frac{dH_k}{dt} \in L_t^1$, we obtain
$$
\int_0^\infty u(t,x) (V(t)-d(x))^2\, dx \in L_t^1(0,\infty).
$$
For $k(x)=\int_0^x d(s)\, ds$, the result follows from the fact that $\frac{dH_k}{dt} \in L_t^1$.
\\[5pt]
{\em Step 2.} 
We claim that for  $z\in [0,\infty)$, we have $|V-d(X(\cdot,z))|^2 \in W_t^{1,1}(0,\infty)$. Consequently, $V(t)-d(X(t,z))$ tends to zero as $t\to \infty$ irrespective of $z\in [0,\infty)$. 

We first prove that $|V-d(X(\cdot,z))|^2 $ is integrable. We compute
$$\begin{array}{rl}
&\rho_0 |V(t)-d(X(t,z))|^2 = \displaystyle \int_0^\infty |V(t)-d(X(t,z))|^2 u(t,x)\, dx
\\[15pt]
&\qquad  \le 2 \displaystyle \int_0^\infty |V(t)-d(x))|^2u(t,x)\, dx + 2 \int_0^\infty |d(x)-d(X(t,z))|^2u(t,x)\, dx.
\end{array}$$
The first term is time integrable thanks to the Step~1. For the second term, we use the entropy provided by Lemma~ \ref{lem:entropy1}
$$
\int_0^\infty |d(x)-d(X(t,z))|^2u(t,x)\, dx \le \beta^2 g(t,z)\leq \beta^2 g(0,z) e^{-2\alpha t}.
$$
Then $|V-d(X(\cdot,z))|^2 \in L_t^1(0,\infty)$. 

Next, we consider the derivative of this function. It reads
$$
\frac{d}{dt}(V(t)-d(X(t,z)))^2= 2 \frac{d V}{dt} (V(t)-d(X(t,z)) - 2 d'(X(t;z)) (V(t)-d(X(t,z)))^2.
$$
We notice that the second term above is integrable thanks to the previous part of the proof, while the first term is also integrable as the product of two $L^2_t$ functions, since $\frac{dV}{dt} \in L^2_t$: indeed,
using the Cauchy-Schwarz inequality in~\eqref{VeqA} we get
$$
\left\vert \frac{d V}{dt} \right\vert^2  \le \left( \int_0^\infty u(t,x) |d(x)-V(t)|\, dx \right)^{2}\le \rho_0\;  \int_0^\infty(d(x)-V(t))^2 u(t,x),
$$
and we use here again the Step~1 to conclude this point.
\\[5pt]
{\em Step 3.} 
For any  $z\in [0,\infty)$, we  decompose the mass conservation relation as
\begin{equation}
\label{mass_defect}
M= V(t) +\rho_0 X(t,z) + B(t,z),  \qquad B(t,z):= \int_0^\infty (x-X(t,z)) u(t,x)\, dx.
\end{equation}
Combining the entropy inequality of Lemma~\ref{lem:entropy1} with the Cauchy-Schwarz inequality, we have
$$
\vert B(t,z)\vert \le \sqrt{\rho_0 g(0,z)} e^{-\alpha t}\qquad \forall t\ge 0.
$$
Hence for any $z\geq 0,$ 
$$\lim_{t\to \infty} V(t) +\rho_0 X(t,z) =\lim_{t\to\infty} d(X(t,z)) + \rho_0 X(t,z) =M.$$ 
By \comJuan{the strict} {monotonicity of the function $x\to d(x)+x$, and since $d(0)<V_0\leq M,$ there exists a unique}  solution $\bar x$ to
$$
d(\bar x) +\rho_0 \bar x=M.
$$
This implies that $\lim_{t\to \infty} X(t,z)=\bar x$, which in turns implies $\lim_{t\to \infty} V(t)=d(\bar x),$ and proves the point~\ref{xbar:1} of the theorem.
\\[5pt]
{\em Step 4.} 
To prove the point~\ref{xbar:3},  we write 
$$\begin{array}{ll}
W_2 (u,\rho_0\delta_{\bar x}) &\leq \sqrt{\iint \vert x-y\vert^2 \rho_0 \delta_{\bar x} (y) u(t,x) dx} \leq \sqrt{\int \vert x-\bar x\vert^2 \rho_0  u(t,x) dx}
\\ \\
& \leq \sqrt{2 \int \vert X(t,z)-\bar  x\vert^2 \rho_0  u(t,x) dx
+ 2 \int \vert X(t,z)- x\vert^2 \rho_0  u(t,x) dx}
\\ \\
& \leq C\left(e^{-\alpha t} + \vert X(t,z)-\bar x\vert \right) \to_{t\to\infty} 0.
\end{array}$$
 In order to obtain an exponential  rate of convergence for the second term $X(t,z)$, we use (\ref{mass_defect}), and writing $M=d(\bar x) +\rho_0 \bar x$ we obtain
\begin{equation}
\label{convenient1}
\vert (V(t)-d(\bar x)) +(\rho_0 X(t,z) - \rho_0 \bar x)\vert=\vert B(t,z)\vert \le \sqrt{\rho_0 g(0,z)} e^{-\alpha t} \quad \forall t\ge 0.
\end{equation}
Next we compute
$$
\frac{d}{dt} |X(t,z)-\bar x| = \mbox{sign}\ (X(t,z)-\bar x) \frac{d X(t,z)}{dt}  = \mbox{sign}\ (X(t,z)-\bar x) \left(V(t) - d(X(t,z))\right)
$$
$$
= \mbox{sign}\ (X(t,z)-\bar x) \left(V(t) - d(\bar x)\right)+  \mbox{sign}\ (X(t,z)-\bar x) \left(d(\bar x) - d(X(t,z))\right)
$$
$$
= \mbox{sign}\ (X(t,z)-\bar x) \left(V(t) - d(\bar x)\right) - d'(\theta) |X(t,z)-\bar x| \:,
$$
for some $\theta \in [0,\infty)$. We rewrite the first term above as
$$
\mbox{sign}\ (X(t,z)-\bar x) \left(V(t) - d(\bar x)\right)
$$
$$
=\mbox{sign}\ (X(t,z)-\bar x) \{\left(V(t) - d(\bar x)\right) +\rho_0 (X(t,z)-\bar x)\} - \rho_0 |X(t,z)-\bar x|\:. 
$$
Thus, using (\ref{convenient1}), we obtain successively
$$
\frac{d}{dt} |X(t,z)-\bar x|\le -(\rho_0+\alpha) |X(t,z)-\bar x| + \sqrt{\rho_0 g(0,z)} e^{-\alpha t}\:,
$$
$$
0 \le |X(t,z)-\bar x| \le |z-\bar x| e^{-t (\rho_0+\alpha)} +\sqrt{\frac{g(0,z)}{\rho_0}} e^{-t (\rho_0+\alpha)} \left(e^{\rho_0 t} -1\right)\:.
$$
Finally, we obtain 
$$
|X(t,z)-\bar x| = O(e^{-\alpha t}) \quad \mbox{and hence}�\quad W_2(u,\rho_0 \delta_{\bar x})= O(e^{-\alpha t}).
$$
Going back to (\ref{convenient1}) we also conclude that
$$
|V(t)-d(\bar x)| = O(e^{-\alpha t}).
$$
\end{proof}
\subsection{Lifshitz-Slyozov with nucleation}
\label{sec:LSN}
We turn to the proof of Theorem~\ref{th:probB:asymp2},  where the Lifshitz-Slyozov system is complemented by a nucleation term~\eqref{LSN}. 

We first recall a simple lemma, variant of Gronwall's lemma, which is used several times in the proofs below.
\begin{Lemma}
Let $f\in {\cal C}^1(\R_+,\R_+)^+,$ $g\in {\cal C}(\R_+)$ and $h\in {\cal C}(\R_+,\R_+)$ satisfy
$$
\f{df}{dt}\leq-h(t)f(t) +h(t)g(t),\qquad h(t)\geq C>0, \quad \lim_{t\to \infty} g(t)=0.
$$
Then $\lim_{t\to\infty} f(t)=0.$
\label{lem:Gronwall}
\end{Lemma}
\begin{proof}
Since $\f{d}{dt} \big(f(t) e^{\int_0^t h(s) ds}\big) \leq h(t)g(t)e^{\int_0^t h(s) ds},$ integrating, we get
$$\begin{array}{ll} f (t) &\leq f(0) e^{-\int_0^t h(s) ds}+ \int_0^t h(s) g(s) e^{-\int_s^t h(u) du}
\\ 
&\leq f(0) e^{-Ct}+ \int_0^{\f{t}{2}} h(s) g(s) e^{-\int_s^t h(u) du} +\int_{\f{t}{2}}^{t} h(s) g(s) e^{-\int_s^t h(u) du}
\\ 
& \leq f(0) e^{-Ct} + \Vert g\Vert_{L^\infty} e^{-\int_{\f{t}{2}}^t h(u)du}
+ \sup_{s\in (\f{t}{2},t)} g(s) \to_{t\to\infty} 0.
\end{array}$$
\end{proof}


As a first  step we show that the number of fibrils increases without bound, as provided by the following statement.
\begin{Lemma}
\label{lem:planar}
Under the assumptions of Theorem~\ref{th:probB:asymp2} we have  
$$
 \rho(t):=\int_0^\infty u(t,x)dx \nearrow +\infty\quad \mbox{and} \quad V(t)\longrightarrow d(0)\quad \mbox{as}\ t\to \infty.
$$
\end{Lemma}
\begin{proof}
Integrating the equation for $u$ in~\eqref{LSN} and using the boundary condition at $x=0$, we have
$$
\f{d\rho}{dt} = \int_0^\infty \f{\p}{\p t} u(t,x) dx=-\int_0^\infty \f{\p}{\p x} \left(V(t)-d(x)\right)u(t,x)dx=V(t)^{i_0}\geq d(0)^{i_0} \geq 0
$$
  so that $\rho$ is increasing towards a limit $0<\rho_\infty\leq \infty$. 
Let us assume that it tends to a finite limit $\rho_\infty<0$ and argue by contradiction. Using Assumption~\eqref{as:d:inc} and the equation for $V$, we get
$$\left(d(0)-V(t)\right) \rho(t) + \alpha \left(M-V(t)\right)\leq \f{dV}{dt} \leq \left(d(0)-V(t)\right) \rho(t) + \beta \left(M-V(t)\right)
$$
and since $d(0)<V(t)\leq M$ and $\rho_0\leq \rho\leq \rho_\infty$
$$\left(d(0)-M\right) \rho_\infty \leq \f{dV}{dt} \leq \left(d(0)-V(t)\right) \rho_0 + \beta M,
$$
so that $V\in {\cal C}^1_b(\R_+).$ Since $V(t)^{i_0}=\f{d\rho}{dt},$ we also have $V^{i_0}\in L^1(\R_+),$ which combined with $V\in {\cal C}^1_b(\R_+)$ implies $\lim_{t\to\infty} V=0.$ Turning to the  double inequality for $V$, this implies
$$
0<d(0) \rho_\infty + \alpha M \leq \liminf_{t\to\infty} \f{dV}{dt} \leq \limsup_{t\to\infty} \f{dV}{dt} \leq d(0) \rho_\infty + \beta M,
$$
which contradicts that $V(t) \to 0.$ Hence $\lim_{t\to\infty} \rho (t)=+\infty.$
Since we have
$$
\f{dV(t)}{dt} \leq \left(d(0)-V(t)\right) \rho(t) + \beta \left(M-V(t)\right),
$$
which implies $\lim_{t\to\infty} V(t)=d(0)$ by applying Lemma~\ref{lem:Gronwall} to $f=V(t)-d(0)$, $h=\rho(t)$ and $g(t)=\f{\beta M}{\rho(t)}\to_{t\to\infty} 0.$
\end{proof}
Note that $V(t) \to d(0)$  also implies that 
\begin{equation}
\lim_{t\to\infty} \int_0^\infty x\, u(t,x)\, dx \rightarrow M-d(0).
\label{nucleation:direc}
\end{equation}
To prove the concentration of polymerized mass at zero, we then prove that  the second moment of $u(t,x)$ vanishes as stated in the following lemma.
\begin{Lemma}
\label{lem:2moment_vanishes}
Under the assumptions of Theorem~\ref{th:probB:asymp2} we have $\displaystyle \lim _{t\to\infty} M_2(t) =0.$
\end{Lemma}
\begin{proof}
We compute 
$$\begin{array}{rl}
 \displaystyle \frac{dM_2(t)}{dt}& =  \displaystyle \int_0^\infty x (V(t)-d(0) + d(0) - d(x)) u(t,x)\, dx
\\[15pt]
& \le \big(V(t)-d(0) \big) M - 2 \alpha M_2(t).
\end{array}
$$
Once again, we apply Lemma~\ref{lem:Gronwall} with $f=M_2,$ $h=2\alpha$ and $g=\f{1}{2\alpha}(V(t)-d(0))M$ which tends to $0$ thanks to Lemma \ref{lem:planar}.
\end{proof}
Lemma~\ref{lem:2moment_vanishes} together with~\eqref{lem:planar} implies the weak convergence of $xu(t,x)$ towards $\left(M-d(0)\right) \delta_0$ as stated in Theorem~\ref{th:probB:asymp2}: let $\phi\in {\cal C}_b^1(\R_+)$ a test function, we have for any $\ep>0$ small enough
$$\begin{array}{ll}
\vert \int_0^\infty x u(t,x) (\phi(x)-\phi(0)) dx\vert &=\int_\ep^\infty +\int_0^\ep x u(t,x) \vert \phi(x)-\phi(0)\vert dx 
\\ \\
&\leq \f{1}{\ep}\Vert \phi \Vert_{L^\infty} M_2(t) + \Vert \phi'\Vert_{L^\infty} M_2(t) \to_{t\to\infty} 0,
\end{array}
$$
and since $\int_0^\infty xu(t,x)dx \to M-d(0),$ we have the desired general convergence result of Theorem~\ref{th:probB:asymp2}. 

 We now prove the rates of convergence of $V$ and $\rho$ according to whether $d(0)>0$ or $d(0)=0.$
One of the key points is to relate the divergence rate of $\rho$ to the convergence rate of $V$, as stated in the following lemma.
\begin{Lemma}
\label{lem:balance}
Under the assumptions of Theorem~\ref{th:probB:asymp2}, we have
$$
\lim_{t\to \infty} \rho(t) (V(t)-d(0)) = d'(0)(M-d(0)).
$$
\end{Lemma}
\begin{proof}
Let us first notice that the previous convergence results also
 imply that 
\begin{equation}
\label{0concentration}
\lim_{t\to \infty} \int_0^\infty (d(x)-d(0))u(t,x)\, dx = d'(0)(M-d(0)).
\end{equation}

We define $w(t)= \rho(t) \big ( V(t)- d(0)\big) -d'(0) \big( M-d(0)\big)$ and compute
$$\begin{array}{rl}
\displaystyle \frac{d w(t)}{dt} & = V^{i_0} \left(V(t)-d(0)\right) +\rho(t) \big(\int_0^\infty d(x) u(t,x) dx - V(t) \rho(t)\big)
\\ \\
&=V^{i_0} \left(V(t)-d(0)\right) 
+\rho(t) \big(\int_0^\infty (d(x)-d(0)) u(t,x) dx - (V(t) -d(0))\rho(t)\big)
\\ \\&= B(t) - \rho(t)  w(t) +\rho(t) C(t),
\end{array}$$
with
$$B(t) = V(t)^{i_0}  \big ( V(t)- d(0)\big), \qquad C(t) = \displaystyle \int_0^\infty \big[ d(x)-d(0) -d'(0) x \big] \; u(t,x) \, dx.
 $$
We then apply Lemma~\ref{lem:Gronwall} with $f=w,$ $h=\rho$ and $g=\f{B(t)}{\rho(t)} + C(t),$ since according to~\eqref{0concentration} and to the weak convergence result, we have
$
C(t)    \underset{ t \to \infty }{\longrightarrow} 0.
$ Hence $\lim_{t\to 0} w(t)=0,$ which proves the result.
\end{proof}
Using Lemma~\ref{lem:balance} allows to focus on the asymptotic rate of divergence of $\rho,$ from which the rate of convergence for $V(t)-d(0)$ follows. 
The two following lemmas now treat respectively the cases  $d(0)>0$ and $d(0)=0$.
\begin{Lemma}
Under the assumptions of Theorem~\ref{th:probB:asymp2}, assume moreover $d(0)>0$.  Then the number of fragments $\rho(t)$ satisfies $$
  \displaystyle\lim_{t\to\infty} \frac{\rho(t)}{t} = d(0)^{i_0}.$$
\label{lem:nfragments_diverge}
\end{Lemma}
\begin{proof}
As already seen, we have
\begin{equation}\label{rho_B}
0< d(0)^{i_0}< \f{d\rho}{dt} = V(t)^{i_0}=\big(d(0)+V(t)-d(0)\big)^{i_0} \leq d(0)^{i_0} + C\big(V(t)-d(0)\big),
\end{equation}
which implies the result since $\lim_{t\to\infty} V(t)-d(0)=0.$
\end{proof}
\begin{Lemma}\label{lem:d0>0}
Under the assumptions of Theorem~\ref{th:probB:asymp2}, assume moreover $d(0)=0$.  Then the number of fragments $\rho(t)$ satisfies
 $$
  \displaystyle\lim_{t\to\infty} \frac{\rho(t)}{t^{\f{1}{i_0+1}}} = (1+i_0)^{\f{1}{i_0+1}}\left(d'(0)M\right)^{\f{i_0}{i_0+1}}.$$
\end{Lemma}

\begin{proof}[Proof of Theorem \ref{th:probB:asymp2},\emph{\ref{th:pbB:extra_info:d(0)=0}}]: 
With the information of \eqref{nucleation:direc} and of Lemma~\ref{lem:balance} we can determine the asymptotic behavior of $\rho$ and $V$ in turn. Given $0<\epsilon<d'(0)M$, we use Lemma \ref{lem:balance} to find $T>0$ such that 
$$
|V(t)\rho(t)-d'(0)M|<\eps,\quad \forall t>T.
$$
We decompose now 
$$
\frac{d\rho}{dt} = V(t)^{i_0} \chi_{(0,T)} + \frac{1}{\rho(t)^{i_0}}V(t)^{i_0}\rho(t)^{i_0}\chi_{(T,\infty)}
$$
for $t>T$, so that
$$
V(t)^{i_0} \chi_{(0,T)} + \frac{(d'(0)M-\epsilon)^{i_0}}{\rho(t)^{i_0}}\chi_{(T,\infty)} < \frac{d\rho}{dt} < V(t)^{i_0} \chi_{(0,T)} + \frac{(d'(0)M+\epsilon)^{i_0}}{\rho(t)^{i_0}}\chi_{(T,\infty)}.
$$
Multiplying by $(1+i_0) \rho(t)^{i_0}$ and integrating in time on $(T,t)$,
$$
\rho^{i_0+1}(T) + (1+i_0)(t-T)(d'(0)M-\epsilon)^{i_0}< \rho^{i_0+1}(t) <\rho^{i_0+1}(T) + (1+i_0) (t-T)(d'(0)M+\epsilon)^{i_0}
$$
with
$$
\rho(T)=\rho(0)+\int_0^T V(\tau)^{i_0}\, d\tau.
$$
Therefore, as $\epsilon$ is arbitrary,
$$
\lim_{t\to \infty} \frac{\rho(t)}{t^{1/(i_0+1)}} =(1+i_0)^{\f{1}{i_0+1}}\left(d'(0)M\right)^{\f{i_0}{i_0+1}}
$$
so that
$$
 \lim_{t\to \infty} t^{1/(i_0+1)} V(t) = \left(\frac{d'(0)M}{1+i_0}\right)^{1/(i_0+1)}.
$$
This concludes the proof of Lemma~\ref{lem:d0>0} and of Theorem~\ref{th:probB:asymp2}, \emph{\ref{th:pbB:extra_info:d(0)=0}}.
\end{proof}
%
\begin{Remark}
Using computations like those in the proof of Lemma \ref{lem:2moment_vanishes} we may deduce that the second moment $M_2(t)$ vanishes at least as fast as $V(t)-d(0)$ does.
\end{Remark}

\section{Fragmentation as a possibly stabilizing secondary process}
\label{sec:proof_fragmentation}


This section is devoted to the proofs of Theorem~{\ref{th:probB:asymp2:2}} and~\ref{th:steady}, which consider the two opposite cases, respectively increasing or decreasing ratio polymerization/depolymerization, \emph{i.e.} decreasing or increasing total growth rate $V(t)-d(x).$


\subsection{Increasing depolymerization rate}
\label{sec:frag:instab}

We prove here  Theorem~\ref{th:probB:asymp2:2}. 
As expected, under our assumptions on the fragmentation rate, the same asymptotic as in Theorem~\ref{th:probB:asymp2} holds but is still faster, thus simplifying rather than complexifying the proof. 
\begin{proof}[Proof of Theorem~\ref{th:probB:asymp2:2}]
We recall that the total mass conservation still holds, $M$ is constant in \eqref{eq:mass},  because fragmentation is conservative in mass, but increases the number of polymers. 
\\[2pt]
{\em First step: $d(0)<V(t)\leq M.$}
This comes from Lemma~\ref{lem:probA:exist},
 whose proof remains valid in the present case.
\\[2pt]
{\em Second step:  $\rho(t)\geq \rho(0) e^{B_m t}.$}
Integrating the equation, and using that $V(t)>d(0),$ we obtain the result since
$$
\frac{d\rho}{dt} = \ep V(t)^{i_0}+ \int_0^\infty B(x)u(t,x)\, dx\geq B_m \rho(t).
$$
\\
{\em Third step: $0<V(t)-d(0)\leq Ct e^{-B_m t}$ }
 We have
$$\begin{array}{rl}
\displaystyle \frac{d}{dt} \big( V(t)- d(0) \big)& = \int_0^\infty (d(x)-d(0)) u(t,x)\, dx - \rho(t) (V(t)-d(0))
\\[5pt]
& \leq \beta M -\rho(t)(V(t)-d(0))
\end{array},
$$
so that
$$\begin{array}{ll}
 V(t)- d(0) &\leq \big( V_0- d(0) \big) e^{-\int_0^t \rho(s)ds } +\beta M \int_0^t e^{-\int_s^t \rho(\sigma)d\sigma } ds 
\\ \\ &\leq \big( V_0- d(0) \big) e^{-\int_0^t \rho(0) e^{B_m s}ds } +\beta M \int_0^t e^{-\int_s^t \rho(0) e^{B_m \sigma}d\sigma } ds
\\ \\&\leq\big( V_0- d(0) \big) e^{-\f{\rho(0)}{B_m} (e^{B_m t}-1) } +\beta M \int_0^t e^{-\f{\rho(0)}{B_m} (e^{B_m t}-e^{B_m s}) } ds.
\end{array}
$$
Next, we use the following convexity  inequality
$$
\f{e^{B_m t}-1}{t}\leq \f{e^{B_m t} - e^{B_m s}}{t-s}, \qquad 0\leq s \leq t , 
$$
to obtain
$$\begin{array}{ll}
 \int_0^t \exp\biggl(-\f{\rho(0)}{B_m} (e^{B_m t}-e^{B_m s})ds \biggr) ds & \leq  
\int_0^t \exp\biggl(-\f{\rho(0)}{B_m} (e^{B_m t}-1)(1-\f{s}{t})ds \biggr) ds
\\[5pt]
&\leq \f{tB_m}{ \rho(0) (e^{B_m t}-1)}.
\end{array}$$
We deduce that $V(t)$ converges to $d(0)$ with the rate $te^{-{B_m t}}$ when $t\to \infty$. 

To conclude the proof of Theorem~\ref{th:probB:asymp2:2}, it remains to show the concentration of the polymerized mass at zero size, which follows from the vanishing the second moment, as in the proof of Theorem~\ref{th:probB:asymp2}.
\end{proof}

%



\subsection{Depolymerization as a stabilizing mechanism}
\label{sec:steady}
A particular feature of Theorem~\ref{th:probB:asymp2:2} is that the results remain unchanged if we take $d(x)\equiv 0$; then $\rho$ still increases exponentially, and since $\f{dV}{dt}=-\rho(t) V(t)$,  $V(t)$ decays to $0$ faster than exponentially. Then,  we can use the fragmentation term, with $c$ in \eqref{as:frag},
to prove that 
$$
\f{dM_2}{dt} \leq - cB_m M_2 +2V(t) M,
$$
which also implies the exponential decay  of $M_2 (t)$. This observation initially motivated our study: one of the most frequently used model for protein polymerization, namely the growth-fragmentation model, with or without nucleation, leads to an asymptotic state of dust rather than of stable large polymers. An increasing depolymerization rate maintains this asymptotic behavior. We  see in Theorem~\ref{th:steady} that, surprisingly, a decreasing depolymerization rate, meaning that large polymers are more stable than small ones, is able to stabilize the system.

Theorem~\ref{th:steady} thus states the { existence of a positive steady state of}  the model  \eqref{eq:mass}--\eqref{bcB} with \emph{decreasing} depolymerization rates. 
In this case, the presence or absence of nucleation does not really play a role in the long-term dynamics, since we  prove below that at $x=0$ we  have $d(0)>\bar V$, the steady state of $V(t)$.

The proof  of Theorem~\ref{th:steady} follows the lines for instance of Theorem~4.6. in~\cite{BP} and also used  and detailed e.g. in~\cite{CDG,DG} for studies of the eigenvalue problem for the growth-fragmentation equation.  We decompose the proof in several steps which are stated as additional theorems.  Indeed, { we encounter a specific difficulty here because the growth rate   $V-d(x)$ vanishes at some nonnegative point $x_0$, so that the solution operator is not easily defined near $x_0$. Indeed, a vanishing transport speed may generate a Dirac solution, however the positive absorbtion term near $x_0$ has a regularizing effect on the solution, which belongs to a $L^p$ space. } For that reason, when solving the regularized problem, we first consider the operator only for $x>d^{-1}(V),$ and then extend it for smaller $x$ by successive use of the Banach-Picard fixed point theorem. 

\begin{Theorem}\label{th:KR}
Assume that $d$ satisfies \eqref{as:d},  that \eqref{customary} and \eqref{as:frag}--\eqref{Bandk} hold true. 

For $\ep>0$, $R>0$ and a given $V\in (\inf_x d(x), d(0))$, setting $x_0=d^{-1}(V)$ and $x_\ep=x_0+\ep$, there exists a unique couple  
$$
(\lambda_V^{\ep,R},U_{V}^{\ep,R})\in \R\times {\cal C}([x_\ep,R],\R_+^*)
$$
solution to the following eigenvalue problem on $[x_\ep,R]:$
\begin{equation}\label{eq:truncated}
\left\{\begin{array}{l}
\f{\p}{\p x} (( V - d(x))U_V^{\ep,R}) +\lambda_{V}^{\ep,R} U_V^{\ep,R}+B (x) U_V^{\ep,R} =2\int_x^R B(y) U_V^{\ep,R} (y)\kappa (y,x) dy,
\\ \\
(V-d(x_\ep))U_V^{\ep,R}{(x_\ep)}=\ep\int_{x_\ep}^R U_V^{\ep,R} (y) dy,\quad U_V^{\ep,R}(x)>0,\quad \int_{x_\ep}^R U_V^{\ep,R} (x)dx=1.
\end{array} \right.\end{equation}
\end{Theorem}
\begin{proof}
We follow the proof of~\cite{BP} Theorem~6.6. or~\cite{DG} Theorem~3 for instance, assuming $B$ and $\kappa$ continuous, otherwise a standard regularization procedure is implemented. For a given $f\in {\cal C}([x_\ep,R],\R_+^*)$, we  apply first the Banach-Picard theorem to the operator 
$$
T_f: {\cal C}([x_\ep,R]) \to {\cal C}([x_\ep,R]), \quad m \to n=T_f(m)
$$
defined thanks to the equation for $n$
$$\left\{\begin{array}{l} 
\mu n + \f{\p}{\p x}( ( V - d(x))n )+B (x) n =2\int_x^R B(y) m (y)\kappa (y,x) dy +f(x),
\\[5pt]
(V-d(x_\ep))n({x_\ep})=\ep\int_0^R m(y) dy.
\end{array}\right.$$
For $\mu>0$ large enough,  $T_f$ is indeed a strict contraction and thus $T_f$ has a unique fixed point. Then we may apply the 
 Krein-Rutman theorem to the operator $f\mapsto n$ with $T_f(n)=n,$ for which strong positivity, continuity and compactness in ${\cal C}([x_\ep,R])$ follows from arguments in the papers mentioned above.
\end{proof}

\begin{Theorem}
\label{th:truncated}
With the assumptions and notations of Theorem~\ref{th:KR}, 
\begin{itemize}
\item  as $\ep\to 0$, we obtain a  weak solution $(\lambda_V^R,U_V^R)\in \R\times L^\infty((x_0,R),\R_+)$ to the limit system of~\eqref{eq:truncated} with $\ep\to 0.$

\item The function $V\mapsto \lambda_V^{R}$ is continuous, and for \comJuan{$V$ in a neighbourhood of $d(0)$ and sufficiently large $R$}, we have $\lambda_V^{R} >0$. 

\item
For $V$ such that $\lambda_V^R> -B(x_0)+d'(x_0),$ we can  extend $U_V^R$ to $[0,R]$ and obtain, up to  renormalisation, a nonnegative solution, for $x\in (0,R)$,  of
\end{itemize}
$$\left\{\begin{array}{c}
\f{\p}{\p x} (( V - d(x))U_V^{R}) +\lambda_{V}^{R} U_V^{R}+B (x) U_V^{R} =2\int_x^R B(y) U_V^{R} (y)\kappa(y,x) dy, 
\\ \\
U_V^{R}(x)\geq 0,\qquad \int_{0}^R U_V^{R} (x)dx=1.
\end{array}\right.$$

\end{Theorem}
\begin{proof} We present a brief proof because most of the fundamental estimates are proved later with a uniform dependency on $R$ that is not needed here.
\\[5pt]
{\em Step 1.}
The proof for the limit is the same as for instance in~\cite{DG}, end of Appendix B. 
In the compact $[x_0,R]$, we can prove uniform bounds in $L^\infty$ for $U_V^{\ep,R}$, take a subsequence converging weakly in $L^\infty([x_0,R])$, and the equation taken in a weak sense follows from the strong convergence of the coefficients in $L^1$. Moreover, we have $(V-d(x))U_V^{R}\in W^{1,\infty}([x_0,R])$ thanks to our assumptions on the coefficients so that $U_V^R$ is continuous on $(x_0,R)$.
\\[5pt]
{\em Step 2.}
The continuity of $V \mapsto \lambda_V^{R}$ can be obtained following the proof of Lemma 3.1. in~\cite{M2}. If $V\to d(0),$ we have $x_0\to 0:$ we are back to the usual - truncated - growth-fragmentation equation for which the strict positivity of the eigenvalue for sufficiently large $R$  has been established~\cite{DG,BP}. 
\\[5pt]
{\em Step 3.}
We restrict now ourselves to \comJuan{the set of values $V$ such that}
$\lambda_{ V}^{R} >-B (x_0)+d'(x_0)$, 
\comJuan{which holds true} 
at least for $x_0=d^{-1}(V)$ small enough (\comJuan{that is, $V\sim d(0)$}) following { the step 2. This is necessary to extend the solution at $x=x_0$ as we see it right below.} For such an $x_0$,
we extend our solution $U^{R}_{ V}$ to $[0,x_0]$ by applying successively a Banach-Picard fixed point theorem  to intervals $[x_0-(k+1)\delta, x_0-k\delta]$ on spaces ${\cal C}\big( [x_0 -(k+1)\delta, x_0-k\delta] \big)$, $k\in \N$, for $\delta$ small enough so that the following operator $m\to n$ is a contraction:
$$\left\{\begin{array}{l}
\f{\p}{\p x} (( V-d(x))n) +\lambda_{V}^{R} n  +B (x) n =2\int_x^{x_0 -k\delta} B(y) m (y)\kappa(y,x) dy +F_k(x),
\\[10pt]
{ n(x_0-k\delta)=U_{ V}^{R} (x_0-k\delta)}, 
\end{array} \right.$$
with $F_k$ defined in the previous steps of the iteration by
$$
F_k(x):=\int_{x_0-k\delta}^R U_{ V}^{R} (y) B(y) \kappa(y,x) dy>0.
$$
{ In these iterations, the most difficult point is the extension for $x<x_0$ close to the singular point  $x_0$, which is possible thanks to the absorption stemming from the choice $\lambda_{ V}^{R} >-B (x_0)+d'(x_0)$.} \comJuan{Let us show that there is some $\delta>0$ small enough such that the map $m\rightarrow n$ is contractive in ${\cal C}\big( [x_0 -\delta, x_0] \big)$, which provides the aforementioned extension}. {  Denoting $n=n_1-n_2$ the difference between the solutions for $m_1$ and $m_2$ respectively, and $m=m_1-m_2,$ and $I(m)=2\int\limits_x^{x_0} B(y)m(y)\kappa(y,x)dy,$ we get, for any $\bar x<x<x_0$
$$\biggl(n(x)\exp (\int_{\bar x}^x \f{\lambda +B(s)-d'(s)}{V-d(s)}ds\biggr)'=\exp (\int_{\bar x}^x \f{\lambda +B(s)-d'(s)}{V-d(s)}ds) \f{I(m)(x)}{V-d(x)},
$$
so that by integration between $x$ and $x_0$ we have
$$\begin{array}{l}n(x)\exp (\int_{\bar x}^x \f{\lambda +B(s)-d'(s)}{V-d(s)}ds)=
 -\int\limits_x^{x_0}
\exp (\int_{\bar x}^s \f{\lambda +B(\sigma)-d'(\sigma)}{V-d(\sigma)}d\sigma) \f{ I(m)(s)}{V-d(s)}ds.
\end{array}
$$
Let us now do an asymptotic expansion around $x_0$ to check the validity of the integral: we have 
$$\f{\lambda +B(x)-d'(x)}{V-d(x)}\sim_{x_0} \f{\lambda+B(x_0)-d'(x_0)}{-d'(x_0)(x-x_0)}=\f{\alpha}{x-x_0},$$
with $\alpha>0$ thanks to our assumption $\lambda +B(x_0)-d'(x_0)>0,$ hence
$$\exp (\int_{\bar x}^{x} \f{\lambda +B(s)-d'(s)}{V-d(x)}  ds) \sim \vert \f{x_0-x}{x_0-\bar x} \vert^\alpha, $$
and finally we get
$$\vert n(x) (\f{x_0-x}{x_0-\bar x})^\alpha \vert \leq C\| m\|_{L^\infty} \int\limits_x^{x_0} (x_0-s)^{\alpha-1}ds,$$
with $C$ only depending on the parameters of the problem.} 

\comJuan{
Once we have the expansion for the exponential term we plug it into the former equality to get
$$\vert n(x) (\f{x_0-x}{x_0-\bar x})^\alpha \vert \sim C_1 \int\limits_x^{x_0} \vert \f{x_0-s}{x_0-\bar x} \vert^\alpha\frac{I(m)(s)}{V-d(s)}ds \sim \frac{C_2}{|x_0-\bar x|^\alpha} \int\limits_x^{x_0} |x_0-s|^{\alpha-1}I(m)(s)ds$$
where by the same token we used that $V-d(s) \sim -d'(x_0) (s-x_0)$. Hence
$$
|n(x)| \le \frac{C_3\|m\|_\infty}{|x_0-x|^\alpha}  \int\limits_x^{x_0} |x_0-s|^{\alpha-1}\int\limits_s^{x_0} B(y)\kappa(y,s)dy\,ds
$$
$$
\le \frac{C_3\|m\|_\infty B_M}{|x_0-x|^\alpha}  \int\limits_x^{x_0} \int\limits_x^{y} |x_0-s|^{\alpha-1} \kappa(y,s)ds\,dy\:.
$$
Here we use \eqref{Bandk} to bound the innermost integral. We consider first the case $0<\alpha\le 1$: there holds that
$$
\int\limits_x^{x_0} \int\limits_x^{y} |x_0-s|^{\alpha-1} \kappa(y,s)ds\,dy \le C \int\limits_x^{x_0} \frac{|y-x_0|^{\alpha-1}|y-x|^\gamma}{y^\gamma}\, dy
$$
and as a consequence
$$
|n(x)| \le  \frac{C_3\|m\|_\infty B_M}{|x_0-x|^\alpha x^\gamma} \frac{|x-x_0|^\alpha}{\alpha} |x_0-x|^\gamma \le C_4 \delta^\gamma \|m\|_\infty.
$$
This shows that the map $m\rightarrow n$ is contractive in ${\cal C}\big( [x_0 -\delta, x_0] \big)$ for $\delta$ small enough. When $\alpha >1$ we argue in a similar fashion that 
$$
\int\limits_x^{x_0} \int\limits_x^{y} |x_0-s|^{\alpha-1} \kappa(y,s)ds\,dy \le C \int\limits_x^{x_0} \frac{|x_0-x|^{\alpha-1}|y-x|^\gamma}{y^\gamma}\, dy
$$
which in the end allows us to recover the same type of contractivity estimate.
}

Finally, we renormalize $U_V^{R}$ by multiplication to achieve 
$\int_0^R U_V^{R} (x)dx=1.$ 
\end{proof}


\begin{Theorem}\label{th:estim}
With the assumptions and notations of Theorem~\ref{th:truncated}, for $R$ large enough there exists $\bar V \in (\inf_x d(x), d(0))$ such that $\lambda_{\bar V}^R=0$, that means 
\begin{equation}
\label{VeqE:steady:trunc}
 \bar V  = \int_{0}^R d(x)U_R(x)dx, \qquad \int_0^R U_R(x) dx=1,
\end{equation}
\begin{equation}
\label{UeqE:steady:trunc}
\frac{\partial}{\partial x} ((\bar V-d(x)) U_R)  = -B(x) U_R(x)   + 2 \int_x^R B(y) \kappa (y,x) U_R(y)\, dy.
\end{equation}
Moreover, the following estimates, independent of $R>0$ large enough, hold.
\begin{enumerate}
\item For $k\geq 0,$ there exists $C_k>0$  such that 
$$
\int_0^R B(x) x^k U_R(x) dx \le C_k.
$$
\item With $\gamma$ defined in the assumption~\eqref{Bandk}, there exists $C>0$ such that for any $x\in (0, R)$ we have
$$
\vert \bar V - d(x) \vert U_R(x) \leq C, \qquad \vert \bar V - d(x) \vert U_R(x) \vert x-x_0\vert^{-\gamma} \leq C .
$$
\item There exists $\eta>0$ such that
$$
 \inf_{x\geq 0} d(x)<d(\infty) +\eta \leq \bar V \leq d(0)-\eta < d(0).
 $$
\end{enumerate}
\end{Theorem}
\begin{proof}
{\em Step 1. A lower bound on $V$ such that $\lambda_V^R\geq 0$.}
\\
In this step, we assume that $\lambda_V^R\geq 0$. On the one hand, integrating the equation against the weight $x$ gives
$$
V\geq \int_0^R d(x)U_V^{R} (x) dx,
$$
which in turns implies (dropping superscript $R$ and subscript $V$)
$$
V\geq \int_0^A d(x) U(x) dx\geq d(A) (1-\int_A^R U(x)dx) \geq d(A)\big(1-\int_A^R \f{x^k}{A^k} U(x)dx\big).
$$

On the other hand, integrating the equation against the weight $x^k$ gives (see details below concerning this a priori estimate),
$$
\int_0^R x^k B(x) U(x) dx \leq C_k
$$
so that, using assumption~\eqref{Brates:2}, for $A >{\cal A}$, we have 
$$
\int_A^R \f{x^k}{A^k} U(x)dx \leq \int_A^R \f{x^k}{A^k} \f{B(x)}{B_m} U(x)dx\leq \f{C_k}{A^k B_m}.
$$

Finally,  we  use~\eqref{as:d}  so that 
$$
 V \geq \big[ d(\infty) +\f{C_d}{A^n} \big] \big[ 1-  \f{C_k}{A^k B_m} \big] .
$$
Now, we choose a value  $k>n$, then choosing $A$ large enough, we obtain that for some $\eta >0$ (depending only on our assumptions on the coefficients $b$, $d$, $\kappa$ but not on $R$), we have
\begin{equation}
\label{lower_V}
 V \geq d(\infty) +\eta > \inf_{x\geq 0} d(x).
\end{equation}
In particular, this inequality holds true for $\bar V$ as built below. 
\\[5pt]
{\em Step 2. Existence of the steady state $(\bar V, U_{\bar V}^R)$.}
\\
We have already seen that $\lambda_V^R$ is a continuous function, positive for $V$ sufficiently close to $d(0)$. Therefore,  to prove that $\lambda_V^R$ vanishes,  it remains to prove that $\lambda_V^{R}$ may be negative.
Assume by contradiction, that for any $V\in (\inf d, d(0))$ we have $\lambda_V^R\geq 0,$  and take $V \to \inf_x d(x)$ in Theorem~\ref{th:truncated} - It is possible to extend $U_V^R$ on $[0,R]$ since the assumption $\lambda_{ V}^{R} >-B (x_0)+d'(x_0)$ is then satisfied for all $V$. We use  the estimate~\eqref{lower_V}  to obtain a contradiction and thus conclude the existence of $\bar V$.
\\

We recall that integrating Equation~\eqref{UeqE:steady:trunc} successively against $1$ and $x$ yields
$$
0\leq (d(0) - \bar V) U(0)=\int_0^R B(x) U_R(x) dx \leq B_M,
$$
$$
\inf_{x\geq 0} d(x) +\eta \leq \bar V = \int_0^R d(x) U_R(x) dx \leq d(0),
$$
so that we find \eqref{VeqE:steady:trunc} and $d(0)\geq \bar V$ and a boundary condition at $x=0$ is not needed.
\\[5pt]
\noindent{\em Step 3. Moments of $B(x)U(x)$.}
\\
 We drop the index $R$ for simplicity. 
Integrating Equation~\eqref{UeqE:steady:trunc}  against $x^k$ yields, using the notation in~\eqref{as:frag},
$$
-k \int_0^R (\bar V - d(x)) x^{k-1} U(x) dx =  - \int_0^R x^k B(x) U(x) a_k (x) dx 
$$
with $a_k(x) \geq c>0$. Hence for $A>{\cal A}$ large enough, $\cal A$ being defined in \eqref{Brates:2}
$$\begin{array}{rl}
c \int_0^R B(x) x^k U(x) dx &\leq  k \bar  V \int_0^R x^{k-1} U(x) dx
 \\ \\
& \le k d(0) \int_0^A x^{k-1} U(x)\, dx + \frac{1}{B_m A}\int_0^R B(x) x^k U(x) dx
\end{array}
$$
and thus
$$
\left( c  -\frac{1}{B_mA} \right) \int_0^R B(x) x^k U(x) dx \le k d(0) \int_0^A x^{k-1} U(x)\, dx \le k d(0) A^{k-1},
$$
so that finally for $A$ large enough
$$
\int_0^R B(x) x^k U(x) dx \le \frac{B_m k d(0) A^k}{c B_m A-1}=C_k.
$$
\\[5pt]
\noindent{\em Step 4. $L^\infty$ bound for $(\bar V - d(x))U(x)$ and, for $x_0<x_{\min}$ defined in~\eqref{Bandk}, for $(\bar V - d(x))U(x) \vert x-x_0\vert^{-\gamma}$.}
\\
We integrate~\eqref{UeqE:steady:trunc} between $x_0:=d^{-1}(\bar V)$ and $x$ to find
$$
\vert \bar V - d(x) \vert U(x) +\big \vert \int_x^{x_0} B(y)U(y)dy \big\vert = 2 \big\vert \int_x^{x_0} \int_z^R B(y) U(y) \kappa(y,z) dy dz \big \vert
$$ 
so that
$$
\vert \bar V - d(x) \vert \, U(x) \leq 2 \big \vert \int_x^{x_0} \int_z^R B(y) U(y) \kappa(y,z) dy dz \big \vert \leq 2B_M.
$$

As right above, we integrate between $x_0<x_{\min}$ and $x$ to find
$$\begin{array}{rl}
\vert \bar V - d(x) \vert \, U(x) &\leq \vert 2\int_x^{x_0} \int_z^R B(y) U(y) \kappa(y,z) dy dz\vert 
\\[15pt]
& \leq 2 C  \int_0^R B(y) U(y) y^{-\gamma} \vert x-x_0\vert^\gamma dy 
\end{array}
$$
thanks to~\eqref{Bandk}, and thus
\begin{equation}
\label{estU_closex0}
\vert \bar V - d(x) \vert \; U(x) \leq  2C \Vert B(y)y^{-\gamma}\Vert_{L^\infty}  \vert x-x_0\vert^\gamma .
\end{equation}
\\
{\em Step 5.  Upper bound for $\bar V$ and bounds on $x_0$.} 
Because we have already proved the lower bound~\eqref{lower_V} on $\bar V$, we know that  $x_0:=d^{-1}(\bar V)$ remains bounded from above. Ensuring now an upper bound strictly smaller than $d(0)$ for $\bar V$, uniformly as $R \to \infty$, will ensure a lower bound for $x_0$. Let us assume $x_0<x_{\min}$.

We have
$$\bar V \leq  d(0) \int_0^\delta U(x)dx + d(\delta) \int_\delta^R U(x)dx= d(\delta) + (d(0)-d(\delta)) \int_0^\delta U(x)dx,$$
so that it remains  to show that the mass does not concentrate only around zero. 


Using assumption~\eqref{as:d:dec}, for $x\in (0, x_{\min})$, we have $|d'(x)| \geq L$ for some constant $L$. Therefore, we have
$$
L \; \vert x-x_0\vert \, U(x) \leq  \vert \bar V - d(x)\vert \, U(x) \leq C \vert x-x_0\vert^\gamma
$$ 
where we have used~\eqref{estU_closex0}. We conclude that
$$
U(x)\leq \vert x-x_0\vert^{\gamma-1} \in L^1((0, x_0+1)).
$$
Hence for some $\delta$ small enough we have $\int_0^\delta U(x)dx <\f{1}{2}$ so that
$$\bar V \leq \f{d(\delta)+d(0)}{2} < d(0),$$
and $x_0>\min(x_{\min},d^{-1}(\f{d(\delta)+d(0)}{2} ))$ remains away from zero.
\end{proof}

The estimates of Theorem~\ref{th:estim} are sufficient to let $R\to\infty$  (see once again~\cite{BP} Theorem 4.6. or the estimates for the proof of ~\cite{DG} Theorem 1). { Thanks to the second estimate, the function $U_R$ is uniformly bounded by $\f{C  \vert x-x_0\vert^\gamma}{\bar V-d(x)}$ which is locally integrable around $x_0$. Therefore,  together with the first estimate which guarantees that  $x^k U_R$ is uniformly in $L^1$ for large $x,$ we infer that $\int x^k U_R(x)dx$ is bounded in $L^1$. Thanks to the Dunford-Pettis theorem, we conclude that $(U_R)$ belongs to a  weak compact set of $L^1$. We can extract a subsequence  which converges $L^1-$ weak towards a certain $U.$  Finally, we can apply the chain rule to the equation (see Eq. (28) in~\cite{DG}) and find that for $k\geq 1$ we conclude that $(x^k U_R)$ is bounded in $W^{1,1},$ so that the strong convergence is proved.}
\\

This concludes the proof of Theorem~\ref{th:steady}.

\section*{Conclusion}

We have investigated, in a systematic way, the polymer size distributions resulting from the interplay between polymerization, depolymerization, fragmentation and nucleation. 

 For the { variant of the classical} Lifshitz-Slyozov model { considered in this paper,} the proofs of convergence are based on two entropy inequalities, which unfortunately fail to be satisfied when adding fragmentation into the system.  Then, the existence proof of a steady state for the growth-decay-fragmentation equation relies on \emph{a priori} estimates, in the same spirit as for the eigenvalue problem of the growth-fragmentation equation, but with the delicate question of a change in the sign of the transport rate.

Our last system shows that { a  steady distribution of fibrils} may  be obtained by a polymerizing-depolymerizing-fragmenting system. Another possibility would be to consider the second term in the asymptotic development for the polymerization and depolymerization reactions, making a diffusion term appear~\cite{ColletHariz}. Indeed \cite{Laurencot2004}  proved the existence of steady states for diffusion-fragmentation equations. Up to our knowledge, this is however the only existing study in this direction. 

Concerning the applications to biology, our study was motivated by \emph{in vitro} experiments of fibril formation~\cite{Radford}, since  most protein fibrils formed \emph{in vitro}  appear to be stable. However, the  stability of fibrils may also be linked to a kind of pseudo-stability, in the spirit of metastable states for Becker-D\"oring~\cite{Penrose1989}, \emph{i.e.} where a very slow degradation would underlie the observed stability : a numerical study could help to answer this more quantitative question.

Further studies could be to generalize the assumptions on the coefficients, and find assumptions which would guarantee uniqueness of a steady state.  Still more challenging is the question of the stability of the steady state, which is also an open problem for the so-called prion system. 

\subsubsection*{Acknowledgments}
Marie Doumic was supported and Juan Calvo was partially supported by the ERC Starting Grant SKIPPER$^{AD}$ (number 306321). Juan Calvo also acknowledges support from``Plan Propio de Investigaci\'on, programa 9'' (funded by Universidad de Granada and european FEDER funds), Projects MTM2014-53406-R, MTM2015-71509-C2-1-R (funded by MINECO and european FEDER funds) and Project P12-FQM-954 (funded by Junta de Andaluc\'ia). Beno\^it Perthame was supported by the ERC Advanced Grant Adora (740623). We thank Romain Yvinec and Wei-Feng Xue for inspiring discussions.


%

\end{document}